\documentclass[a4paper,12pt]{amsart}
\usepackage[english]{babel}
\usepackage{amsmath,amssymb,amsthm}

\usepackage{graphicx}

\usepackage[pdftex]{color}
\usepackage[bookmarks=true,hyperindex,pdftex,colorlinks, citecolor=blue]{hyperref}
\hypersetup{pdftitle={Bishop-Phelps-Bollob\'{a}s version of Lindenstrauss properties A and B},
pdfauthor={Sun Kwang Kim and  Han Ju Lee},pdfpagemode=UseOutlines}

\theoremstyle{plain}
\newtheorem{theorem}{Theorem}[section]
\newtheorem{corollary}[theorem]{Corollary}

\newtheorem{proposition}[theorem]{Proposition}
\newtheorem{lemma}[theorem]{Lemma}

\theoremstyle{definition}

\newtheorem{example}[theorem]{Example}

\newtheorem{definition}[theorem]{Definition}

 \DeclareMathOperator{\re}{Re\,}

\newcommand{\inner}[1]{\ensuremath{\left\langle #1\right\rangle}}

\newcommand{\eps}{\varepsilon}

\renewcommand{\leq}{\leqslant}
\renewcommand{\geq}{\geqslant}

\begin{document}

\title[Retraction and Bishop-Phelps-Bollob\'as theorem]{Simultaneously continuous retraction and Bishop-Phelps-Bollob\'as type theorem}

\author[Kim]{Sun Kwang Kim}
\address[Kim]{Department of Mathematics, POSTECH, Pohang (790-784), Republic of Korea}
\email{\texttt{lineksk@gmail.com}}

\author[Lee]{Han Ju Lee}
\address[Lee]{Department of Mathematics Education,
Dongguk University - Seoul, 100-715 Seoul, Republic of Korea}
\email{\texttt{hanjulee@dongguk.edu}}


\date{January 29th, 2014}

\subjclass[2000]{Primary 46B20; Secondary 46B04, 46B22}

\keywords{Banach space, approximation, retraction, norm-attaining operators, Bishop-Phelps-Bollob\'{a}s theorem.}

\thanks{The second named author partially supported by Basic Science Research Program through the National Research Foundation of Korea(NRF) funded by the Ministry of Education, Science and Technology (2012R1A1A1006869).}

\begin{abstract} The dual space $X^*$ of a Banach space $X$ is said to admit a uniformly simultaneously continuous retraction if there is a retraction $r$ from $X^*$ onto its unit ball $B_{X^*}$ which is uniformly continuous in norm topology and continuous in weak-$*$ topology. 

We prove that  if a Banach space (resp. complex Banach space) $X$ has a normalized unconditional Schauder basis  with unconditional basis constant 1 and if $X^*$ is uniformly monotone (resp. uniformly complex convex), then  $X^*$ admits a uniformly simultaneously continuous retraction. It is also shown that $X^*$ admits such an retraction  if $X= \left[\bigoplus X_i\right]_{c_0}$ or $X=\left[\bigoplus X_i\right]_{\ell_1}$, where $\{X_i\}$ is a family of separable Banach spaces whose  duals are uniformly convex with moduli of convexity $\delta_i(\eps)$ with $\inf_i \delta_i(\eps)>0$ for all $0<\eps<1$.

Let $K$ be a locally compact Hausdorff space and let $C_0(K)$ be the real Banach space consisting of all real-valued continuous functions vanishing at infinity. As an application of simultaneously continuous retractions, we show that a pair $(X, C_0(K))$ has the Bishop-Phelps-Bollob\'as property for operators if $X^*$ admits a uniformly simultaneously continuous  retraction. As a corollary, $(C_0(S), C_0(K))$ has the Bishop-Phelps-Bollob\'as property for operators for every locally compact metric space $S$.
\end{abstract}

\maketitle

\section{Introduction}

Let $X$ be a real or complex Banach space and $A$ be a subset of $X$. A continuous function $r:X\to A$ is said to be a {\it retraction} if $r$ is the identity on $A$. Retractions have various applications in nonlinear  geometric functional analysis  \cite{Ben1, Ben2, BenLin}. Benyamini introduced the notion of simultaneously continuous retraction from  dual space $X^*$ onto $B_{X^*}$. More precisely, the dual space $X^*$ of a Banach space $X$ is said to {\it admit a (resp. uniformly) simultaneously continuous retraction} if there is a retraction $r$ from $X^*$ onto $B_{X^*}$ which is both weak-$*$ continuous and norm continuous (resp. uniformly norm-continuous).
Benyamini \cite{Ben1} showed, in particular, that $E^*$ admits uniformly simultaneously continuous retraction if $E^*$ is a separable uniformly convex space, or $E$ is the space $C(K)$ of all real-valued continuous functions on a compact metric space $K$.

 As remarked in Proposition~4.22.~\cite{BenLin}, there is a connection between simultaneously continuous retractions and the denseness of norm attaining operators into $C(K)$. In this paper, we deal with the existence of uniformly simultaneous continuous retraction in a certain Banach space and its applications to Bishop-Phelps-Bollob\'as type theorem.

\section{Uniformly simultaneously continuous retraction}

Let $\{e_j\}$ be a normalized unconditional Schauder basis for $X$ with unconditional basis constant 1. Its biorthogonal functionals will be denoted by  $\{e_j^*\}$. In fact, it is easy to see that $X$ and $X^*$ are Banach lattices and, for every $x^*\in X^*$, we have
\[ x^* = \text{weak}*\sum_{j=1}^\infty x^*(j)e_j^*,\] where $x^*(j) = \inner{x^*, e_j}.$ 
 Recall that a Banach lattice $X$ is uniformly monotone if, for all $\eps>0$,
\[M(\eps) = \inf\{ \| |x| + |y| \|-1 : \|x\|=1, \|y\|\ge \eps\}>0.\] It is easy to check that $\eps \mapsto M(\eps)$ is a monotone increasing function and $M(\eps)\le \eps$ for all $\eps>0$. This $M$ is called the {\it modulus of monotonicity} of $X$. It is easy to check that if $X$ is uniformly monotone, then $X$ is strictly monotone. That is, $\| |x| +  |y| \|>\|x\|$ for all $x\in X$ and for all nonzero element $y$ in $X$. The uniform monotonicity of Banach lattice is equivalent to the uniform complex convexity of its complexification \cite{Lee, Lee2}. The complex convexity has been used to study density of norm-attaining operators between Banach spaces \cite{Acosta-RACSAM, ChoiKamLee}.

Benyamini showed \cite{Ben1} that if $X$ has a shrinking Schauder basis $\{e_j\}$ with $\{e_j^*\}$ beging strictly monotone, then $X^*$ admits a simultaneously continuous retraction. It is also shown that for $X=\ell_p$, $1\le p<\infty$ or $X=c_0$,  $X^*$ admits a uniformly simultaneously continuous retraction.

For $t\ge 0$, we define $M^{-1}(t) =\sup\{ \eps\ge 0: M(\eps)\le t\}$ for a monotone increasing function $M$. The modulus of continuity for a function $\varphi$ is defined by
\[ \omega_{\varphi}(t) = \sup\{ \|\varphi(x^*) - \varphi(y^*)\| : \|x^* - y^*\|\le t\}.\]
Let $f$ be a nonnegative function on a deleted neighborhood of $0$ with $\lim_{t\to 0+} f(t)=0$. We say that $X^*$ admits {\it a $f$-uniformly simultaneously continuous retraction} if there is a uniformly simultaneously continuous retraction $\varphi$ with $\omega_\varphi(t)\le f(t)$.

\begin{theorem}\label{thm:monotone}
Suppose that a Banach space $X$ has a normalized unconditional Schauder basis $\{e_j\}$ with unconditional basis constant 1. If $X^*$ is uniformly monotone with modulus $M$, then $X^*$ admits a uniformly simultaneously continuous retraction with modulus of continuity $2M^{-1}$.
\end{theorem}
\begin{proof}
Notice that $X^*$ is uniformly monotone and it is order-continuous (cf. \cite{Lee}) and $\{e_j^*\}_{j=1}^\infty$ is a Schauder basis. 
Given $x^*=  \sum_{j=1}^\infty a_j e_j^*$ with $x^*\not\in B_{X^*}$, there is a unique $n$ so that
\[\left\|\sum_{j=1}^{n-1} a_j e_j^*\right\|<1, \ \ \ \text{and} \ \ \ \left\|\sum_{j=1}^{n} a_j e_j^*\right\|\ge 1.\] By the strict monotonicity and convexity of norm,  there is a unique $0<t\le 1$ so that
\[  \left\|\sum_{j=1}^{n-1} a_j e_j^* + ta_n e_n \right\|=1,\] and we define $\varphi(x^*) = \sum_{j=1}^{n-1} a_j e_j^* + ta_n e_n$. Defining $\varphi$ as an identity on $B_{X^*}$, we first show that $\varphi$ is uniformly norm continuous.

Notice that if $\|x^*\|\ge1$, then by the construction of $\varphi$ and uniform monotonicity,
\[ \|x^*\|= \| |\varphi(x^*)| + |x^*-\varphi(x^*)| \|\ge 1 + M( \|x^* - \varphi(x^*)\|)\]
and we have $M(\|x^* - \varphi(x^*)\|) \le \|x^*\|-1$. That is,
\[\|x^* - \varphi(x^*)\| <M^{-1}( \|x^*\|-1).\]
We claim that for all $x^*, y^*$ in $X^*$,
\[\|\varphi(x^*) -\varphi(y^*) \| \le 2M^{-1}(\|x^*-y^*\|). \]
Because $M(\eps)\le \eps$ for all $\eps>0$, we have $M^{-1}(t) \ge t$ for all $t>0$. Hence
this inequality is trivial if  $\|x^*\|\le 1$ and  $\|y^*\|\le 1$. If $\|x^*\|>1$ and $\|y^*\|\le 1$, then
\begin{align*}
\|\varphi(x^*) -\varphi(y^*) \|& =  \|\varphi(x^*)-y^*\|\le \|\varphi(x^*) - x^*\| + \|x^*-y^*\|\\
&\le M^{-1}(\|x^*\|-1) + \|x^*-y^*\|\\
&\le M^{-1}(\|x^*\|-\|y^*\|) + M^{-1}(\|x^*-y^*\|)\\
&\le 2M^{-1}(\|x^*- y^*\|).
\end{align*}
We assume that $\|x^*\|>1$ and $\|y^*\|>1$ and write
\[ \varphi(x^*) = \sum_{j=1}^{n-1} x^*(j) e_j^* + t x^*(n)e_n^* \ \ \ \text{and}  \ \ \ \varphi(y^*) = \sum_{j=1}^{m-1} y^*(j) e_j^* + sy^*(m)e_m^*\]
where $x^*(i)=x^*(e_i)$ and $y^*(i)=y^*(e_i)$ for every $i\in \mathbb{N}$.

 For each $n\in \mathbb{N}$, let $P_n$ be a projection on $X$ defined by $P_n(\sum \alpha_ie_i)=\sum_{i=1}^n\alpha_ie_i$ and $P_n^*$ be the adjoint operator.
We may also assume that $m\ge n$ and then $\varphi(x^*)=\varphi(P_m^*(x^*))$, $\varphi(y^*)=\varphi(P_m^*(y^*))$ and $\|P_m^* x^* - P_m^*y^*\|\le \|x^* - y^*\|$  shows that we can replace $x^*$ and $y^*$ by $P_m^*(x^*)$ and $P_m^*(y^*)$ respectively.
That is,
\[x^* = \sum_{j=1}^{m} x^*(j) e_j^*  \ \ \ \text{and}  \ \ \ y^* = \sum_{j=1}^{m} y^*(j) e_j^* .\]
 If $m>n$, $x^*(m)$ and $y^*(m)$ can be replaced by $sy^*(m)$ and $x^*(m)- y^*(m) +sy^*(m)$ without changing $\|x^*-y^*\|$, $\varphi(x^*)$ and $\varphi(y^*)$. On the other hand, if $n=m$ and $t\le s$, then, letting
\[x^*_1 = \sum_{j=1}^{n-1} x^*(j) e_j^* + s x^*(n)e_n^*,\] we have $\|x_1\|\ge1$, $\varphi(x^*)=\varphi(x_1^*)$ and $\|x_1^* - \varphi(y^*)\|\le \|x^* - y^*\|$. Hence we get
\begin{align*}
 \|\varphi(x^*) - \varphi(y^*) \|& = \|\varphi(x_1^*) - \varphi(y^*) \|\le  \|\varphi(x_1^*) -x_1^*\|+ \|x_1^* - \varphi(y^*)\|\\
&\le M^{-1}(\|x_1^*\|-1) + \|x^*-y^*\|\\
&\le 2M^{-1}(\|x^*-y^*\|),
\end{align*}
since $\|x_1^*\|\le \|\varphi(y^*) \| + \|\varphi(y^*) - x_1^*\|\le 1+  \|\varphi(y^*) - x_1^*\|$ and $t\le M^{-1}(t)$ for all $t\ge 0$.

Now, we will show that $\varphi$ is weak-$*$ continuous. Suppose that a net $\{x_\alpha^*\}$ converges weak-$*$ to $x^*$. Since the range of $\varphi$ is bounded and $X$ has the Schauder basis $\{e_j\}$, it is enough to check that $\lim_{\alpha} \inner{\varphi(x_\alpha^*), e_j} = \inner{\varphi(x^*), e_j}$ for all $j$.
Given $x^*\in X^*$, suppose first that there exists a unique $n$ such that
\[ \|\sum_{j=1}^{n-1} x^*(j) e_j^* \|<1 \ \ \ \ \text{and} \ \ \ \ \ \|\sum_{j=1}^{n} x^*(j) e_j^* \|\ge 1\] and
$ \varphi(x^*) = \sum_{j=1}^{n-1} x^*(j) e_j^* + tx^*(n) e_n^*$ for some $0<t\le 1$. Since $P_{n-1}^*(x^*_\alpha)$ converges to $P_{n-1}^*(x^*)$ in norm, it is clear that $\lim_{\alpha} \inner{ \varphi(x_\alpha^*), e_j} =\lim_{\alpha} \inner{x_\alpha^*, e_j}= x^*(j)$ for each $1\le j\le n-1$. Hence, we may assume that $\|P_{n-1}^*(x_\alpha^*)\|<1$ for all $\alpha$. We claim that $\lim_\alpha \varphi(x_\alpha^*)(j)=0=\varphi(x^*)(j)$ for all $j\ge n+1$. Otherwise, there exist a $j_0\ge n+1$, a subnet $(x_\beta^*)$ and an $\eps_0>0$ such that $|\varphi(x_\beta^*)(j_0)|\ge \eps_0$ for all $\beta$.
Then
\[ \eps_0 \le |\varphi(x_\beta^*)(j_0)|\le |x_\beta^*(j_0)|\to |x^*(j_0)|.\] Hence $\|P_{j_0}^*(x^*)\|>1$ and we may assume that $\|P_{j_0}^*(x_\alpha^*)\|>1$. So there exist $n\le n_\beta \le j_0$ such that for some $0<t_\beta \le 1$,
\[\varphi(x_\beta^*) = \sum_{j=1}^{n_\beta-1} x_\beta^*(j) e_j^* + t_\beta x_\beta^*(n_\beta) e_{n_\beta}^*. \]
Since $\varphi(x_\beta^*)(j_0)\neq 0$, we have $j_0 \le n_\beta$. So $n_\beta = j_0$ for all $\beta$. We may assume that $\lim_\beta t_\beta = t_0$. Then
\[ 1=\lim_\beta \|\varphi(x_\beta^*)\| =  \lim_\beta \left\|\sum_{j=1}^{j_0-1} x_\beta^*(j) e_j^* + t_\beta x_\beta^*(j_0) e_{j_0}^*\right\|=  \left\|\sum_{j=1}^{j_0-1} x^*(j) e_j^* + t_0 x^*(j_0) e_{j_0}^*\right\|.\]
Because $j_0\ge n+1$, we get $t_0=0$, which is a contradiction to  that $|t_{\beta}x_\beta^*(j_0)| = |\varphi(x_\beta^*)(j_0)|\ge \eps_0$ for all $\beta$.

We have only to show that $\lim_\alpha \varphi(x_\alpha^*)(n) =\varphi(x^*)(n)=tx^*(n)$. If $\|x^*_\alpha\|\le 1$ or $x_\alpha(n)=0$, then set $t_\alpha=1$.  If $\|x^*_\alpha\|>1$ and $x_\alpha^*(n)\neq 0$, then choose $0\le t_\alpha \le 1$ so that $\varphi(x_\alpha^*)=t_\alpha x^*_\alpha(n)$. So we have for all $\alpha$, $\varphi(x_\alpha^*)(n)=t_\alpha x_\alpha^*(n)$.
Notice that if $t_\alpha<1$, then
\[ \varphi(x_\alpha^*) = \sum_{j=1}^{n-1} x_\alpha^*(j)e_j^* + t_\alpha x_\alpha^*(n)e_n^*.\]

For any subnet $(x_\gamma)$, we can find a further subnet $(x_\beta)$ such that $\lim_\beta t_\beta=t_0$. Suppose first that $t_0<1$. Then we may assume that $t_\beta<1$ for all $\beta$. This means that
\[ 1=\lim_{\beta} \|\varphi(x_\beta^*)\| = \left\|\sum_{j=1}^{n-1} x^* (j)e_j^*+ t_0 x^*(n)e_n^*\right\|.\]
By the strict monotonicity, we get $t_0 = t$ and
\[\lim_\gamma \varphi(x_\gamma^*)(n) = \lim_\gamma t_\gamma x_\gamma^*(n) = tx^*(n) = \varphi(x^*)(n).\]
Secondly, suppose that $t_0=1$. Then we have
\[ 1 =\lim_\beta \|\varphi(x_\beta^*)\| \ge\lim_\beta  \left\|  \sum_{j=1}^{n-1} x_\beta^*(j)e_j^*+ t_\beta x_\beta^*(n)e_n^*  \right\| =  \left\|  \sum_{j=1}^{n-1} x^*(j)e_j^*+ x^*(n)e_n^*  \right\|\ge 1.   \]
This shows that $t=1$ and \[\lim_\beta \varphi(x_\beta^*)(n) =\lim_\beta t_\beta x_\beta^*(n) = x^*(n) =\varphi(x^*)(n).\]
Hence we conclude that $\lim_\alpha \varphi(x_\alpha^*)(n) = \varphi(x^*)(n)$.

Finally, suppose that $\|P_n^*(x^*)\|<1$ for all $n$. So, $\|x^*\|\le 1$. Fix $n\in \mathbb{N}$.  Then there exists $\alpha_n$ such that $\|P_n^*(x_\alpha^*)\|<1$ for all $\alpha\ge \alpha_n$. Hence this shows that
\[\lim_{\alpha} \inner{ \varphi(x_\alpha^*), e_j } = \lim_{\alpha} \inner{ x_\alpha^*, e_j } =  \inner{x^*, e_j} = \inner{\varphi(x^*), e_j}\] for all $j\le n$. Since the equality holds for arbitrary $n$, we get the desired result.
\end{proof}

\begin{example}
It is easy to check that  every $\ell_p$ $(1\le p<\infty)$  is uniformly monotone. There has been an extensive study about the uniform monotonicity of Orlicz-Lorentz spaces (c.f. \cite{FK, HN}).
\end{example}

Recall that the uniform complex convexity is equivalent to the uniform monotonicity on Banach lattices \cite{Lee, Lee2}. Hence we have the following.

\begin{corollary}\label{thm:monotone}
Suppose that a complex Banach space $X$ has a normalized unconditional Schauder basis $\{e_j\}$ with unconditional basis constant 1. If $X^*$ is uniformly complex convex, then $X^*$ admits a uniformly simultaneously continuous retraction.
\end{corollary}

It is observed \cite{Ben1} that if $Y^*$ admits a ($f$-uniformly) simultaneously continuous retraction and $X$ is a norm-one complemented subspace of $Y$, so does $X^*$. Concerning the stability under the direct sum, it is shown that if we take $p_n =1-\frac1n$, and $X=\left[\bigoplus_n \ell_{p_n}\right]_1$, then $X^*$ does not admit a simultaneously continuous retraction. However we get the following affirmative result.

Now, we see some stability results. The following is clear and we omit the proof.

\begin{proposition}
Let $\{X_i\}_{i\in \mathbb{N}}$ be a family of Banach spaces  and let $X= \left[\bigoplus X_i\right]_{c_0}$ or $X=\left[\bigoplus X_i\right]_{\ell_p}$ for  $1\le p<\infty$. If $X^*$ admits a $f$-uniformly simultaneously continuous retraction $\varphi$, then each $X_i$ admits a $f$-uniformly simultaneously continuous retraction.
\end{proposition}

\begin{proposition}\label{prop:product}
Let $\{X_j\}_{j\in J}$ be a family of Banach spaces and let $X= \left[\bigoplus X_n\right]_1$. Suppose that each $X^*_j$ admits a uniformly simultaneously continuous retraction $\varphi_j$. If
\[\lim_{\eps\to 0+}\sup_{j\in J}\omega_{\varphi_j}(\eps)=0,\] then $X^*$ admits a uniformly simultaneously continuous retraction. In particular, a finite $\ell_1$ sum of Banach spaces whose duals admits uniformly simultaneously continuous retractions also admits uniformly simultaneously continuous retraction.
\end{proposition}
\begin{proof}
For each $x^*\in X^*$, define $\varphi(x^*) = (\varphi_{j}(x^*))_{j\in J}$. Then it is easy to check that $\varphi$ is uniformly norm-continuous and weak-$*$ continuous.
\end{proof}

We do not know that the similar result of \label{prop:product} holds for $c_0$ or $\ell_p$ sums for $1<p<\infty$. However, we provide a positive result for separable uniformly smooth spaces. Recall that a Banach space $X$ is said to be uniformly convex if the modulus of convexity 
\[ \delta_X(\eps) = \inf \left\{  1- \left\| \frac{x+y }{2}\right\|  : x, y \in S_X, \ \  \text{and} \ \ \|x-y\|\ge \eps \right\}
\] is positive for all $0<\eps<1$. A Banach space $X$ is uniformly smooth if and only if $X^*$ is uniformly convex. In the proof, we will use the following lemma.

\begin{lemma}[\mbox{\cite[Lemma~3.3]{AAGM2}}]\label{elementary}
Let $\{c_n\}$ be a sequence of complex numbers with $|c_n|\leq 1$ for every $n$, and let $\eta>0$ be such that for a convex series $\sum \alpha_n$, $\re \sum_{n=1}^\infty \alpha_n c_n >1-\eta$. Then for every $0<r<1$, the set $A : = \{ i \in \mathbb{N} : \re c_i > r \}$, satisfies the estimate
\[ \sum_{i\in A} \alpha_i \geq 1-\frac{\eta}{1-r}.\]
\end{lemma}

\begin{theorem}
 Let $X= \left[\bigoplus X_i\right]_{c_0}$, where $X_i$'s are Banach spaces and let $\delta_i(\eps)$ be  modulus of convexity of $X_i^*$.  Suppose that each space $X_i$ is separable  and $\inf_i \delta_i(\eps)>0$  for all $0<\eps<1$.  Then, $X^*$ admits a uniformly simultaneous continuous retraction.
\end{theorem}
\begin{proof}
For each $i\in \mathbb{N}$ there exists a sequence of finite-dimensional subspaces $E^1_i\subset E^2_i\subset E^3_i \subset ...$ such that $\dim E^n_i=n$ and $\bigcup_{n=1}^\infty E_i^{n}$ is dense in $X_i$. Let $e_i$ be the standard basis of $c_0$ which ensures  that $\bigcup_{i=1}^\infty (X_i\otimes e_i)$ is dense in $X$, where 
\[ X_i \otimes e_i = \{ x\otimes e_i : x\in X_i\}.\]

For each $i,j\in \mathbb{N}$, we define a sequence of spaces 
$$E_k=\left(\cup_{p+q<i+j} E^p_q \otimes e_q\right)\cup\left( \cup_{q\leq j} E^q_{i+j-q}\otimes e_{i+j-q}\right)$$ where $k={{(i+j-1)(i+j-2)}\over{2}}+j$. We clearly see that $E_k\subset E_{k+1}$ for every $k\in \mathbb{N}$.

Let $R_k~:~E_k\longrightarrow X$ be a natural embedding (for the convenience, we set  $E_0=\{0\}$ and $R_0~:~\{0\}\longrightarrow X$). By the uniform convexity, it is easy to check that there is a unique Hahn-Banach extension of every element of $E_k^*$ to $X^*$. So let  $H_k~:~{E_k}^*\longrightarrow {X}^*$ be the map defined by the Hahn-Banach extension theorem.

We also define a map $\psi_k~:~{E_k}^*\longrightarrow {E_{k+1}}^*$ by $\psi_k={R_{k+1}}^*\circ H_k$. For each $x^*$, let $n(x^*)=\text{inf}\{k~:~\|{R_k}^*x^*\|\ge 1\}$, where we use the convention that  $\inf \emptyset = \infty$.

We define a retraction $\phi:X^*\longrightarrow B_{X^*}$. If $\|x^*\|\leq 1$, then $\phi(x^*)=x^*$. If $\|x^*\|> 1$ and $n(x^*)=1$, then we put $\phi(x^*)=H_1({R_1}^*x^*/\|{R_1}^*x^*\|)$. We assume that $\|x^*\|>1$ and $n(x^*)>1$. For the convenience we write $n=n(x^*)$. Since ${R_{n}}^* x^*|_{E_{n-1}}=\psi_{n-1}({R_{n-1}}^*x^* )|_{E_{n-1}}=x^* |_{E_{n-1}}$, we have $\|{R_{n-1}}^* x^*\|=\|\psi_{n-1}({R_{n-1}}^*x^* )\|<1$. Hence, there exists a unique $0<\lambda\leq 1$ such that $\|\lambda{R_n}^*x^*+(1-\lambda)\psi_{n-1}({R_{n-1}}^*x^*)\|=1$. We put $\phi(x^*)=H_n(\lambda{R_n}^*x^*+(1-\lambda)\psi_{n-1}({R_{n-1}}^*x^*))$.

We now show that a retraction $\phi$ is weak-$*$ continuous. Suppose that  $(x^*_\alpha)$ converges to $x^*$ in the weak-$*$ topology.

  First assume that $n=n(x^*)<\infty$. Since ${R_{n}}^*x^*_\alpha$ converges to ${R_{n}}^* x^*$ in norm, we have ${R_{n}}^*\phi(x^*_\alpha)$ converges to ${R_{n}}^*\phi(x^*)$ in norm. This implies that every weak-$*$ limit point of a net $(\phi(x^*_\alpha))$ is an extension of ${R_{n}}^*\phi(x^*)$. Since $\|{R_n}^*\phi(x^*)\|=1= \|\phi(x^*)\|$ and the Hahn-Banach extension  is unique, $\phi(x^*_\alpha)$ weak-$*$ converges to $\phi(x^*)$.  
 On the other hand, assume $\|{R_n}^*x^*\|< 1$ for every $n\in \mathbb{N}$. Since the net $(\phi(x^*_\alpha))$ is bounded, we have only to show that $\phi(x_\alpha^*)(x)$ converges to $\phi(x^*)(x)$ for all $x\in E_n$ and for all $n\ge 1$. Fix $N$. Then ${R_N}^*x^*_\alpha$ converges to ${R_N}^*x^*$  in norm and there exists $\alpha_0$ such that  $\|R_N^*x^*_\alpha\|<1$ for all $\alpha>\alpha_0$ and $\phi(x_\alpha^*)$ is an extension of $R_N^*x^*_\alpha$ for all $\alpha>\alpha_0$. That is, $\phi(x^*_\alpha)(x)= ({R_N}^*x^*_\alpha)(x)$ for each $\alpha>\alpha_0$ and $x\in E_N$. Hence $\phi(x^*_\alpha)(x)$ converges to $\phi(x^*)(x)$ for all $x\in E_N$. Because $N$ is arbitrary,  $\phi(x^*_\alpha)$ converges to $\phi(x^*)$ in the weak-$*$ topology.\\

We calculate the norm-modulus of continuity of $\phi$. For $\epsilon>0$, we fix $x^*$, $y^*$ $\in X^*$ satisfying $\|x^*-y^*\|<\delta(\epsilon)^2$, and let $n=n(x^*)\leq n(y^*)=m$. If $n=\infty$, then it is clear. So assume first that $n\le m<\infty$.

Without loss of generality, we assume that $\phi(y^*)$ is an extension of ${R_n}^*y^*$. Indeed, if $n<m$, then this follows from the definition of $\phi$. On the other hand, if $n=m$, then we choose $u^*\in X^*$ which annihilates $E_{n-1}$. Since ${R_n}^*y^*-\psi_{n}({R_{n-1}}^*y^*)$ and ${R_n}^*x^*-\psi_{n}({R_{n-1}}^*x^*)$ both annihilate $E_{n-1}$, we see that they are multiples of ${R_{n}}^*u^*$. This fact and the convexity of $\|\cdot\|$ imply  that there exists $\alpha$ so that either
$$\|{R_n}^*(y^*+\alpha u^*)\|=1 \text{~and~} \|{R_n}^*(x^*+\alpha u^*)\|\geq 1~\text{~or}$$ 
$$\|{R_n}^*(y^*+\alpha u^*)\|\geq 1 \text{~and~} \|{R_n}^*(x^*+\alpha u^*)\|= 1.$$

Hence, we assume $\|{R_n}^*(y^*+\alpha u^*)\|=1$ and $\|{R_n}^*(x^*+\alpha u^*)\|\geq 1$. (otherwise, we change the role of $x^*$ and $y^*$.) We now take $x^*+\alpha u^*$ and $y^*+\alpha u^*$ instead of $x^*$ and $y^*$.\\

For any element $z$ in a space of vector valued sequence like $X$ and $X^*$, we write $z=(z(1),z(2),...)$. Choose $x\in S_{E_n}$ so that ${R_n}^*\phi(x^*)(x)=1$, then we see that $1={{{R_n}^*\phi(x^*)(i)}\over{\|{R_n}^*\phi(x^*)(i)\|}}(x(i))={{{R_n}^*\phi(x^*)(i)}\over{\|\phi(x^*)(i)\|}}(x(i))$ for every $i\in C$, where $C= \{ i : {R_n}^*\phi(x^*)(i)\neq 0\}$.

From the definition of $\phi$, we have $\re {R_n}^*(x^*)(x)\geq 1$, and so 
\begin{align*}
1-\delta(\epsilon)^2
&<\re \big({R_n}^*(x^*)\big)(x)-\|{R_n}^*(x^*-y^*)\|\\
&\leq \re {R_n}^*(y^*)(x)=\sum \re {R_n}^*(y^*)(i)(x(i))
\end{align*}
Define a set $A=\left\{i~:~\re {{{R_n}^*(y^*)(i)}\over{\|\phi(y^*)(i)\|}}(x(i))>1-\delta(\epsilon),~\|\phi(y^*)(i)\|\neq 0\right\}$. Then, Lemma~\ref{elementary} shows that 
$$\sum_{A} \|\phi(y^*)(i)\|>1-\delta(\epsilon),\ \ \text{~and~}\ \ \sum_{A^c} \|\phi(y^*)(i)\|<\delta(\epsilon).$$

Since $\phi(y^*)$ is an extension of ${R_n}^*y^*$, for each $i\in A\cap C$, we get
\begin{align*}
\left\|{{\phi(y^*)(i)}\over{\|\phi(y^*)(i)\|}}+ {{\phi(x^*)(i)}\over{\|\phi(x^*)(i)\|}}\right\|
&\geq{{{R_n}^*(y^*)(i)}\over{\|\phi(y^*)(i)\|}}(x(i))+ {{{R_n}^*\phi(x^*)(i)}\over{\|\phi(x^*)(i)\|}}(x(i))\\
&>2-\delta(\epsilon)
\end{align*}
and so, 
$$\left\|{{\phi(y^*)(i)}\over{\|\phi(y^*)(i)\|}}-{{\phi(x^*)(i)}\over{\|\phi(x^*)(i)\|}}\right\|<\epsilon.$$
Moreover, for each $i\in A\cap C$,
\begin{align*}
\|\phi (y^*)(i)-\phi(x^*)(i)\|
&=\left\|{{\phi(y^*)(i)}\over{\|\phi(y^*)(i)\|}}-{{\phi(x^*)(i)}\over{\|\phi(y^*)(i)\|}}\right\|\|\phi(y^*)(i)\|\\
&<\left(\left\|{{\phi(y^*)(i)}\over{\|\phi(y^*)(i)\|}}-{{\phi(x^*)(i)}\over{\|\phi(x^*)(i)\|}}\right\|+\left\|{{\phi(x^*)(i)}\over{\|\phi(x^*)(i)\|}}-{{\phi(x^*)(i)}\over{\|\phi(y^*)(i)\|}}\right\|\right)\\
&\text{~}~\cdot\|\phi(y^*)(i)\|\\
&<\epsilon \|\phi(y^*)(i)\|+\big|\|\phi(y^*)(i)\|-\|\phi(x^*)(i)\|\big|.
\end{align*}
So we have for all $i\in A$,
\[ \|\phi (y^*)(i)-\phi(x^*)(i)\| <\epsilon \|\phi(y^*)(i)\|+\big|\|\phi(y^*)(i)\|-\|\phi(x^*)(i)\|\big|.\]

On the other hand, the assumption $\|x^*-y^*\|<\delta(\epsilon)^2$ implies that $\|{R_n}^*x^*-{R_n}^*y^*\|<\delta(\epsilon)^2$, and so $\sum \big|\|{R_n}^*x^*(i)\|-\|{R_n}^*y^*(i)\|\big|\leq\sum \|{R_n}^*x^*(i)-{R_n}^*y^*(i)\|<\delta(\epsilon)^2$. Since $\|{R_n}^*y^*(i)\|\leq\|\phi(y^*)(i)\|$, $\|{R_n}^*x^*\|\geq 1$, and $\|\phi(y^*)\|=1$, we have,  setting $P = \{i : \|\phi(y^*)(i)\|\ge  \|{R_n}^*x^*(i)\|\}$ and $Q= \{\|\phi(y^*)(i)\|< \|{R_n}^*x^*(i)\|\}$, 
\begin{align*}
\sum\big|\|\phi(y^*)(i)\|&-\|{R_n}^*x^*(i)\|\big|\\
&= \sum_{P}\left( \|\phi(y^*)(i)\|-\|{R_n}^*x^*(i)\| \right) + \sum_Q \left(  \|{R_n}^*x^*(i)\|- \|\phi(y^*)(i)\|\right)\\
 & =
  1- \sum_{Q} \|\phi(y^*)(i)\|-\sum_P\|{R_n}^*x^*(i)\|  + \sum_Q \left(  \|{R_n}^*x^*(i)\|- \|\phi(y^*)(i)\|\right) \\
  &\le \sum_Q\|{R_n}^*x^*(i)\|- \sum_{Q} \|\phi(y^*)(i)\|  + \sum_Q \left(  \|{R_n}^*x^*(i)\|- \|\phi(y^*)(i)\|\right)\\
  &\le 2\sum_Q \left(  \|{R_n}^*x^*(i)\|- \|R_n^*y^*(i)\|\right) <2\delta(\eps)^2.
\end{align*}
Notice also that  $R_n^*x^*$ and $R_{n-1}^*x^*$ may have only one different term. Suppose that this different term is $n_1$th term of $R_n^*x^*$. Then  $\|R_{n-1}^*x^*(i)\|= \|R_n^*x^*(i) \| = \|\phi(x^*)(i)\|$ for all $i \neq n_1$. Therefore we have  $\sum_{i\neq n_1}\big|\|\phi(y^*)(i)\|-\|\phi(x^*)(i)\|\big|<2\delta(\epsilon)^2$.

 Since $\sum \|\phi(y^*)(i)\|=\sum\|\phi(x^*)(i)\|=1$, we have $\big|\|\phi(y^*)(n_1)\|-\|\phi(x^*)(n_1)\|\big|<2\delta(\epsilon)^2$. Moreover, the fact  that $\sum\big|\|\phi(y^*)(i)\|-\|\phi(x^*)(i)\|\big|<4\delta(\epsilon)^2$ shows
 \begin{align*}
 \sum_{A^c}\|\phi(x^*)(i)\|
 &\leq \sum_{A^c} \|\phi(x^*)(i)-\phi(y^*)(i)\|+ \sum_{A^c}\|\phi(y^*)\|\\
 &< 4\delta(\epsilon)^2+\delta(\epsilon).
 \end{align*}
Hence, we deduce that
\begin{align*}
\|\phi(x^*)-\phi(y^*)\|
&=\sum_A \|\phi (y^*)(i)-\phi(x^*)(i)\|+\sum_{A^c}\|\phi (y^*)(i)-\phi(x^*)(i)\|\\
&\leq \sum_A \epsilon \|\phi(y^*)(i)\|+ \sum_A\big|\|\phi(y^*)(i)\|-\|\phi(x^*)(i)\|\big|\\
&~\text{~}~+\sum_{A^c}\|\phi (y^*)(i)\|+\sum_{A^c}\|\phi (x^*)(i)\|\\
&<\epsilon+4\delta(\epsilon)^2+4\delta(\epsilon)^2+\delta(\epsilon)+\delta(\epsilon)\\
&=\epsilon+8\delta(\epsilon)^2+2\delta(\epsilon).
\end{align*} 

Finally, assume that $n<m=\infty$. In this case, $\|y^*\|\le 1$. If $\|x^*\|\le 1$, then the desired result clearly holds. So assume that $\|x^*\|>1$.  Let $y_t^* =  tx^* + (1-t)x^*$ and let
\[ t_0 = \sup \{ 0<t<1: \|y^*_t\|=1\}.\] It is clear that $0\le t_0<1$. For each $t_0<s<1$, $\|y^*_s\|>1$ and
$\|x^* - y^*_s\| \le  \|x^*- y^*\|<\delta(\eps)^2$. From the previous result, we have
\[ \|\phi(x^*) - \phi(y_s)\| < \eps + 8\delta(\eps)^2 + 2\delta(\eps).\]
Since $y_s^*$ converges to $y_{t_0}$ as $s$ tends to $t_0$, the weak-$*$ continuity of $\phi$ shows that
\[ \|\phi(x^*) -\phi(y_{t_0})\| \le \eps + 8\delta(\eps)^2 + 2\delta(\eps).\] Since $\|y_{t_0}\|\le 1$, we have
\[ \|\phi(x^*) -\phi(y)\| \le \|\phi(x^*) -\phi(y_{t_0})\| + \|y_{t_0}^*  - y^* \|  \le \eps + 9\delta(\eps)^2 + 2\delta(\eps).\]
This completes the proof.
\end{proof}

For $1<p<\infty$, $\ell_p$ sum of a countable family of separable uniformly convex spaces with uniformly lower bounded moduli of convexity is separable uniformly convex \cite{Day}, we get the following.

\begin{corollary}
Let $\{X_i\}_{i\in \mathbb{N}}$ be a family of Banach spaces whose dual spaces are separable uniformly convex with moduli of convexity  $\delta_i(\eps)$ such that $\inf_i \delta_i(\eps)>0$ for all $0<\eps<1$ and let $X= \left[\bigoplus X_i\right]_{c_0}$ or $X=\left[\bigoplus X_i\right]_{\ell_p}$ for $1\le p<\infty$. Then, $X^*$ admits a uniformly simultaneous continuous retraction.
\end{corollary}

\begin{proposition}\label{prop:basic}
Let $L$ be a locally compact Hausdorff space, $K$ be the one-point compactification of $K$ and let $M(L)$ and $M(K)$be the Banach spaces of all scalar-valued Borel regular measures on $L$ and $K$ with the total variational norms, respectively. Suppose that $M(K)$ admits a uniformly simultaneously continuous retraction as  a dual of $C(K)$. Then $M(L)$ admits  a uniformly simultaneously continuous retraction as a dual of $C_0(L)$.
\end{proposition}
\begin{proof}
Let $K= L\cup \{\infty\}$ and let $\phi$ be a $f$-uniformly simultaneously continuous retraction from $C(K)^*$ onto $B_{C(K)^*}$. Then for each $\mu\in M(L)=C(L)^*$ and for each Borel subset $E$ of $K$, define
$\tilde{\mu}(E) = \mu(E\setminus \{\infty\})$. Then it is clear that $\tilde\mu \in M(K).$ Define the map $\psi : M(L) \to B_{M(L)}$ by, for each $f\in C_0(L)$,
\[ \inner{ f, \psi(\mu)}  = \int_L f d \phi(\tilde \mu).\]
Then it is easy to check that $\psi$ is weak-$*$ continuous on $M(L)=C_0(L)^*$ and it is $f$-uniformly continuous with respect to the norm.
\end{proof}

\begin{corollary}
Let $L$ be a locally compact metrizable Hausdorff space. Then the real space $C_0(L)^*$ admits a uniformly simultaneously continuous retraction.
\end{corollary}
\begin{proof}
It is shown that if $K$ is compact metrizable space, then the real space $C(K)^*$ admits a uniformly simultaneously continuous retraction. Since $L$ is metrizable, its one-point compactfication $\hat{L}$ is compact metrizable. Hence the result follows from Proposition~\ref{prop:basic}.
\end{proof}

\section{Retraction and Bishop-Phelps-Bollob\'as property}

The Bishop-Phelps theorem  \cite{BP} states that for a Banach space $X$, every element in its dual space $X^*$ can be approximated by ones that attain their norms. Since then, there has been an extensive research to extend this result to bounded linear operators between Banach spaces  \cite{Bou,JoWo,Lindens, Partington, Schacher, Schacher-Pacific} and non-linear mappings \cite{AAP, AGM, AFW, C, CK1, KL}. On the other hand, Bollob\'as \cite{Bollobas} sharpened the Bishop-Phelps theorem which is called the Bishop-Phelps-Bollob\'as theorem.

\begin{theorem}[Bishop-Phelps-Bollob\'{a}s theorem]\label{thm:BPBTheorem}
Let $X$ be a Banach space. If $x\in S_X$ and $x^*\in S_{X^*}$ satisfy $|x^*(x)-1|< \eps^2/4$, then there exist $y\in
S_X$ and $y^*\in S_{X^*}$ such that $y^*(y)=1$, $\|x^*-y^*\|<\eps$ and $\|x-y\|<\eps$.
\end{theorem}

Acosta, Aron, Garc\'ia and Maestre \cite{AAGM2} introduced the Bishop-Phelps-Bollob\'as property to study
extensions of the theorem above to operators between Banach spaces.

\begin{definition}[\textrm{\cite[Definition~1.1]{AAGM2}}]
A pair of Banach spaces $(X,Y)$ is said to have the
\emph{Bishop-Phelps-Bollob\'{a}s property} (\emph{BPBp} in short) for operators 
if, for every $\eps\in (0,1)$, there is $\eta(\eps)>0$ such that for every
$T_0\in L(X,Y)$ with $\|T_0\|=1$ and every $x_0\in S_X$ satisfying
$$
\|T_0(x_0)\|>1-\eta(\eps),
$$
there exist $S\in L(X,Y)$ and $x\in S_X$ such that
$$
1=\|S\|=\|Sx\|,\qquad \|x_0-x\|<\eps \qquad \text{and} \qquad \|T_0-T\|<\eps.
$$
In this case, we will say that $(X,Y)$ has the BPBp with function $\eps\longmapsto \eta(\eps)$. The pair $(X, Y)$ is said to have the \emph{Bishop-Phelps Property} (BPp) if the set of all norm-attaining operators is dense in $L(X, Y)$.
\end{definition} 

It is clear that BPBp implies BPp. Recall that Bourgain \cite{Bou} showed that $(X, Y)$ has the BPp for every Banach space $Y$ if $X$ has the Radon-Nikod\'ym property.
However, it is shown \cite{AAGM2}  there exists a Banach space $Y$ such that $(\ell_1, Y)$ does not have BPBp even though $\ell_1$ has the Radon-Nikod\'ym property.

In the study of the operators from a Banach space into $C(K)$, the following representation theorem is useful.  We are  stating a version of this representation theorem for operators into $C_0(S)$ space, which is a slight modification of \cite[Theorem 1, p.~490]{DuSch} and we omit the proof.

\begin{lemma}
\label{lem-isometric-isomorphism}
Let $X$ be a Banach space and let $L$ be a locally compact Hausdorff topological space. Given an operator
$T:X\longrightarrow{C_0(S)}$, define $\mu:S\longrightarrow{X^*}$ by $\mu(s)=T^*(\delta_s)$ for
every $s\in S$. Then the relationship
$$
[Tx](s)= \mu(s)(x), \ \ \  \forall x\in X, s \in S
$$
defines an isometric isomorphism between $\mathcal{L}(X,C_0(L))$ and the space of $w^*$-continuous
functions from $S$ to $X^*$ which vanishes at infinity, endowed  with the supremum norm, i.e.
$\Vert{\mu}\Vert=\sup\{\Vert{\mu(s)}\Vert:s\in{S}\}$. The subspace of  compact operators
corresponds to  norm continuous functions which vanishes at infinity.
\end{lemma}

If $C(K)$ is the space of all continuous functions on a compact Hausdorff space $K$ and $X$ is a Banach space whose dual $X^*$ admits a uniform simultaneously continuous retraction, then the norm-attaining operators are dense in the space $\mathcal{L}(X, C(K))$ of bounded linear operators from $X$ into $C(K)$  \cite[Proposition~4.22.]{BenLin}. So $L_\infty[0,1]$ does not admit a uniformly simultaneously continuous retraction because the pair $(L_1[0,1], C(S))$ does not have the BPp for a certain compact metric space $S$ \cite{Schacher-Pacific, JohWol}.  It is worth-while to note that  $(L_1(\mu), L_\infty(\nu))$ has BPp if $\mu$ is any measure and $\nu$ is a localizable measure \cite{FP, PS}. These results are refined to show that  $(L_1(\mu), L_\infty(\nu))$ has BPBp if $\mu$ is any measure and $\nu$ is a localizable measure \cite{ACGM, CKLM}.

Let $f$ be a nonnegative nondecreasing function such that $\lim_{t\to 0+}f(t)=0=f(0)$. A map $\varphi: X^* \to B_{X^*}$ is called an {\it $f$-approximate nearest point map} if $\|\varphi(x^*) - x^*\| \le d(x^*, B_{X^*}) + f(d(x^*, B_{X^*}))$ for all $x^*\in X^*$. This notion is introduced by Benyamini \cite{Ben1}. A dual space $X^*$ is said to {\it admit weak-$*$ approximate nearest point map} if there exists a weak-$*$ continuous $f$-approximate nearest point map $\varphi:X^* \to B_{X^*}$. Notice that the weak-$*$ continuous approximate nearest point map is a weak-$*$ continuous retraction. It is easy to check that if $X^*$ admits a uniformly simultaneously continuous retraction $\varphi: X^*\to B_{X^*}$, then $\varphi$ is a weak-$*$ $\omega_\varphi$-approximate nearest point map  \cite{Ben1}.

\begin{theorem}
Let $K$ be a locally compact Hausdorff space and let $X$ be a Banach space. If $X^*$ admits a weak-$*$ approximate nearest map, then the pair $(X, C_0(K))$ has the BPBp.
\end{theorem}
\begin{proof}
Let $r: X^* \to B_{X^*}$ be a weak-$*$ $f$-approximate nearest point map. Given $\eps>0$, suppose that $\|T(x_0)\|>1-\eps^2/4$ for some $T\in S_{L(E, C(K))}$ and $x_0\in S_E$. Let $\varphi :K\to E^*$ be the function $\varphi(s) = T^*(\delta_s)$ for all $s\in K$. Choose $t_0\in K$ such that $|T(x_0)(t_0)| = |\inner{ x_0, T^*(\delta_{t_0})}| = |\varphi(t_0)(x_0)|>1-\eps^2/4$. By the Bishop-Phelps-Bollob\'as theorem \ref{thm:BPBTheorem}, there exists a norm-attaining functional $x_1^*\in S_{E^*}$ and $x_1\in S_X$ such that
\[ \|x_0 - x_1 \|<\eps, \ \ \ \ \|x_1^* - \frac{\varphi(t_0)}{\|\varphi(t_0)\|}\|<\eps.\]
Since $\|\varphi(t_0) - \frac{\varphi(t_0)}{\|\varphi(t_0)\|} \|= 1-\|\varphi(t_0)\|<\eps^2/4<\eps$, we have $\|x_1^* - \varphi(t_0)\|<2\eps$. Choose a function $f_0 \in C_0(K)$ such that $f_0(t_0)=1$ and $0\le f\le 1$.
Define $\psi: K\to E^*$ by
\[ \psi(t) = r(\varphi(t) + f_0(t)(x_1^* -\varphi(t_0))) \ \ \ \ (t\in K).\]
Then $\psi(t_0) = r(x_1^*)=x_1^*.$ Let $S$ be the corresponding operator and
\[ 1\ge \|S\|\ge\|Sx_1\| \ge |\inner{ Sx_1, \delta_{t_0}}|= |\inner{\psi(t_0), x_1}|=| \inner{x_1^*, x_1}|=1.\]
Then we have
\begin{align*}
 \|S- T\|& =\sup_{t\in K} \|\varphi(t) - \psi(t) \| = \sup_{t\in K} \|\varphi(t)-  r(\varphi(t) + f_0(t)(x_1^* - \varphi(t_0))) \|\\
& \le  \sup_{t\in K} \|(\varphi(t)+f_0(t)(x_1^*-\varphi(t_0)))-  r(\varphi(t) + f_0(t)(x_1^* - \varphi(t_0))) \| + \|x_1^*-\varphi(t_0)\| \\   
&\le d(\varphi(t) + f_0(t)( x_1^*-\varphi(t_0)), B_{X^*}) +f( d(\varphi(t) + f_0(t)( x_1^*-\varphi(t_0)), B_{X^*})) +  2\eps\\
&\le \|x_1^*-\varphi(t_0)\| + f(\|x_1^*-\varphi(t_0)\|) +2\eps\\
&\le 4\eps + f(2\eps).
\end{align*} This completes the proof.
\end{proof}

Cascales, Guirao and Kadets \cite{CasGuiKad} (cf. \cite{ACK}) showed that every Asplund operator $T$ from a Banach space $X$ into a uniform algebra  $A$ can be approximated by norm-attaining Asplund operators. In particular, $(X, C(K))$ has the BPBp if $X$ is an Asplund space. Since $C[0,1]$ is not an Asplund space, the Banach space whose dual admits the uniformly simultaneously continuous retraction need not be an Asplund space. Benyamini also constructed an example which shows that there is a  (Asplund) Banach space which is isomorphic to $\ell_2$ whose dual does not admit a uniformly simultaneously continuous retraction \cite{Ben1}.

\begin{proposition}
Let $\{X_j\}_{j\in J}$ be a family of Banach spaces and let $X= \left[\bigoplus X_n\right]_1$. Suppose that each $X^*_j$ admits a weak-$*$ $f$-approximate nearest point map $\varphi_j$ with a common function $f$. Then $X$ admits a weak-$*$ $f$-approximate nearest point map.
\end{proposition}

Proposition~\ref{prop:product} shows the following.
\begin{corollary}
Let $\{X_j\}_{j\in J}$ be a family of Banach spaces and let $X= \left[\bigoplus X_j\right]_1$. Suppose that each $X^*_j$ admits a uniformly simultaneously continuous retraction $\varphi_j$. If
\[\lim_{\eps\to 0+}\sup_{j\in J}\omega_{\varphi_j}(\eps)=0,\] then $(X, C_0(K))$ has the BPBp for all locally compact Hausdorff spaces $L$.
\end{corollary}

For the range spaces, the stability of the BPBp under various direct sums of Banach spaces is studied in \cite{ACKLM}. We get here some stability results for the domain spaces when the range is $C(K)$.
\begin{example}
Let $X$ be a Banch space whose dual $X^*$ admits a uniformly simultaneously continuous retraction like $\ell_p$ or $C(S)$ spaces for all compact Hausdorff space $S$. Then
($\ell_1(X), C(K))$ has the BPBp for all compact Hausdorff space $K$. Moreover, we also have the same result for the finite $\ell_1$ sums of different Banach spaces  whose dual admits a uniformly simultaneously continuous retraction. For example, we see that $(\ell_p\oplus_1 C(S),C(K))$ has the BPBp.
\end{example}

Recently it is shown \cite{8} that the pair $(C(S), C(K))$ has the BPBp if $C(S)$ and $C(K)$ are spaces of real-valued continuous functions on a  compact Hausdorff spaces $S$ and $K$ respectively. However it is still open for the spaces of complex-valued continuous functions. It is shown \cite{Ben2} that $C(S)^*$ admits weak-$*$ approximate nearest point map if $S$ is a compact metric space.

\begin{corollary}
Let $S$ be a locally compact metrizable space and $K$ a locally compact Hausdorff space. Then for real-spaces $C_0(S)$ and $C_0(L)$, the pair $(C_0(S), C_0(L))$ has the BPBp.
\end{corollary}

It is worth-while to remark that the first-named author shows that $(c_0, X)$ has the BPBp for all uniformly convex spaces $X$ \cite{Kim-c_0}.

Let $K(X, Y)$ be a the subspace of $L(X, Y)$ which consists of all compact operators from a Banach space $X$ into a Banach space $Y$. Recently the notion of $B^k$ was introduced by Mart\'in \cite{Mar}. A Banach space $Y$  is said to have property $B^k$ if for any Banach space $X$, the norm-attaining compact operators are dense in $K(X, Y)$. Johnson and Wolfe  \cite{JoWo} showed that $C(K)$ space has property $B^k$. Using retraction, we get the following theorem which is a generalization of Theorem~2.2 in \cite{AAGM2}.

\begin{theorem}
Let $K$ be a compact Hausdorff space and let $E$ be a Banach space. Then for each $0<\eps<1$, there is $\eta(\eps)>0$ such that if $T\in S_{K(E, C(K))}$ and $\|T(x_0)\|>1-\eps(\eps)$, there exist $S\in S_{K(E, C(K))}$ and $x_1\in S_E$ such that $\|S(x_1)\|=1$, $\|x_0 - x_1\|<\eps$ and $\|S-T\|<\eps$.  In fact we can take $\eta(\eps) = \frac{\eps^2}{64}$.
\end{theorem}
\begin{proof}
Given $\eps>0$, suppose that $\|T(x_0)\|>1-\eps^2/4$ for some $T\in S_{L(E, C(K))}$ and $x_0\in S_E$. Let $\varphi :K\to E^*$ be the function $\varphi(s) = T^*(\delta_s)$ for all $s\in K$. Since $T$ is compact, $\varphi$ is norm-continuous. Choose $t_0\in K$ such that $|T(x_0)(t_0)| = |\inner{ x_0, T^*(\delta_{t_0})}| = |\varphi(t_0)(x_0)|>1-\eps^2/4$. By the Bishop-Phelps-Bollob\'as theorem~\ref{thm:BPBTheorem}, there exists a norm-attaining functional $x_1^*\in S_{E^*}$ and $x_1\in S_X$ such that
\[ \|x_0 - x_1 \|<\eps, \ \ \ \ \|x_1^* - \frac{\varphi(t_0)}{\|\varphi(t_0)\|}\|<\eps.\]
Since $\|\varphi(t_0) - \frac{\varphi(t_0)}{\|\varphi(t_0)\|} \|= 1-\|\varphi(t_0)\|<\eps^2/4<\eps$, we have $\|x_1^* - \varphi(t_0)\|<2\eps$.
Let $r: E^* \to B_{E^*}$ be the retraction defined by
$r(x)=x$ if $\|x\|\le 1$ and $r(x)=\frac{1}{\|x\|}x$ if $\|x\|\ge 1$.
Define the norm-continuous map $\psi: K\to E^*$ by
\[ \psi(t) = r(\varphi(t) + x_1^* -\varphi(t_0)) \ \ \ \ (t\in K).\]

Then $\psi(t_0) = r(x_1^*)=x_1^*.$ Let $S$ be the corresponding compact operator and
\[ 1\ge \|S\|\ge\|Sx_1\| \ge |\inner{ Sx_1, \delta_{t_0}}|= |\inner{\psi(t_0), x_1}|=| \inner{x_1^*, x_1}|=1.\]
Hence we have $\|S\|=1=\|Sx_1\|$. Since $1\le \|y\|\le 1+\eps$ implies that
\[\|r(y^*) - y^*\| \le \|r(y^*) - r(\frac{y^*}{\|y^*\|}) \| +\|\frac{y^*}{\|y^*\|}-y^*\|\le 2\eps,\]
we have
\begin{align*}
 \|S- T\|& =\sup_{t\in K} \|\varphi(t) - \psi(t) \| = \sup_{t\in K} \| r(\varphi(t) + x_1^* - \varphi(t_0)) - \varphi(t) \|\\
& \le 2\eps + \|x_1^*-\varphi(t_0)\|\le 4\eps.
\end{align*} Therefore, by letting $\eta(\eps) =\frac{\eps^2}{64}$, we get the desired result.
\end{proof}

Because $C(K)$ space is a predual of an $L_1$ space, the above theorem is equivalent to the following which is proved in \cite{8} and we omit the proof.
\begin{theorem}\cite{8}
For each $0<\eps<1$, there is $\eta(\eps)>0$ such that if  $E$ is any Banach space, $Y$ is any predual of an $L_1$-space, $T\in S_{K(E, Y)}$ and $\|T(x_0)\|>1-\eps(\eps)$, there exist $S\in S_{K(E, Y)}$ and $x_1\in S_E$ such that $\|S(x_1)\|=1$, $\|x_0 - x_1\|<\eps$ and $\|S-T\|<\eps$.
\end{theorem}

\end{document}